\documentclass[11pt]{amsart}
\usepackage{amsmath,amsfonts,amsthm,amstext,amscd,amssymb,mathtools}

\usepackage{hyperref}
\usepackage{enumerate}
\usepackage{stmaryrd}

\def\cal#1{\mathcal{#1}}
\def\bb#1{\mathbb{#1}}

\def\lr#1{\left\langle #1\right\rangle}
\usepackage[all]{xy}

\def\co{\colon\thinspace}

\newcommand{\barr}{\begin{array}}
	\newcommand{\earr}{\end{array}}
\newcommand{\beqq}{\begin{equation}}
\newcommand{\eeqq}{\end{equation}}
\newcommand{\beao}{\begin{eqnarray*}}
	\newcommand{\eeao}{\end{eqnarray*}\noindent}
\newcommand{\beam}{\begin{eqnarray}}
\newcommand{\eeam}{\end{eqnarray}\noindent}
\newcommand{\bdis}{\begin{displaymath}}
\newcommand{\edis}{\end{displaymath}\noindent}

\newtheorem{theorem}{Theorem}[section]
\newtheorem*{theoremA}{Theorem}
\newtheorem{lemma}[theorem]{Lemma}

\newtheorem{prop}[theorem]{Proposition}
\newtheorem{claim}[theorem]{Claim}

\theoremstyle{definition}

\newtheorem{example}[theorem]{Example}

\theoremstyle{remark}
\newtheorem{rem}[theorem]{Remark}

\begin{document}

\title[Foliations on Euclidean Space]
 {Metric Foliations on the Euclidean Space}

\author[Speran\c{c}a]{Llohann D. Speran\c ca}
\author[Weil]{Steffen Weil}




\begin{abstract}We complete a minor gap in Gromoll and Walschap classification of metric fibrations from the Euclidean space, thus completing the classification of Riemannian foliations on Euclidean spaces.
\end{abstract}

\maketitle

\section{Introduction}

A (non-singular) \textit{Riemannian foliation} is a foliation whose leaves are locally equidistant. A \textit{Riemannian submersion} is a submersion whose fibers are locally equidistant.
Metric foliations and submersions on specific Riemannian manifolds have been studied and classified.
For instance, Lytchak--Wilking \cite{lytchak2016riemannian} complete the classification of Riemannian foliations of the Euclidean sphere;  Gromoll--Walshap \cite{gromoll2001} propose a classification of Riemannian submersions of the Euclidean space  and Florit--Goertsches--Lytchak--T\"oben \cite{florit2015Riemannian} prove that  any Riemannian foliation $\cal F$ of the Euclidean space $\bb R^{n+k}$ is defined by a submersion $\pi\co \bb R^{n+k}\to M^n$ whose fibers coincide with the leaves of $\cal F$.

However, two gaps in \cite{gromoll2001} were pointed out in \cite{weil}, thus reopening the question of the classification of Riemannian submersions/foliations on the Euclidean space. 
More specifically, \cite{florit2015Riemannian} questions:

\vspace{0.2cm}
\noindent \textit{Question 1.10.} Is any Riemannian foliation on the Euclidean space homogeneous?
\vspace{0.2cm}

The purpose of this note is to complete the gaps in \cite{gromoll2001}, answering the question above affirmatively:

\begin{theoremA}
Every Riemannian foliation with connected fibers on the Euclidean space is homogeneous.
\end{theoremA}

In the next sections, we briefly discuss Gromoll--Walschap's proof and present a workaround for the gaps pointed out in \cite{weil}. The new argument happens to be quite elementary and starts just before the first gap, making it easy  to be put together for a complete proof.


%

\vspace{0.2cm}

\paragraph{Acknowledgments}
The authors thank A. Lytchak for his support. The first named author is supported by FAPESP grant number 2017/19657-0 and CNPq grant number 404266/2016-9. He also would like to thank the University of Cologne for the hospitality. 
Part of this work is part of the Masters Thesis of the second named author.

\section{Original proof and gap}\label{gap}

Gromoll--Walschap \cite{gromoll2001} stated the following theorem:

\begin{theorem}[\cite{gromoll2001}, page 234]
	\label{thm:GW}
	Let $\pi:\bb R^{n+k}\to M^n$ be a Riemannian submersion of the Euclidean space with connected fibers. Then
	\begin{enumerate}
		\item there is a fiber $F$ (over a soul of $M$) which is an affine subspace of the Euclidean space, that, up to congruence, may be taken to be $F=\bb R^k\times \{0\}$.
		\item there is a representation $\phi:\bb R^k\to {SO}(n)$ such that $\pi$ is the orbit fibration of the free isometric group action $\psi$ of $\bb R^k$ on
		$\bb R^{n+k}=\bb R^k\times \bb R^n$ given by
		\begin{equation*}
		\psi(v)(u,x)=(u+v, \phi(v)x),\quad u,v\in\bb R^k,\, x\in \bb R^n.
		\end{equation*}
	\end{enumerate}
\end{theorem}

As a first step, \cite{gromoll1997metric} proves that the fiber $\pi^{-1}(b)=F$ over the soul $\{b\}$ of $M$ is totally geodesic, 
concluding item $(1)$ in Theorem \ref{thm:GW}.
 
Recall that the \emph{integrability tensor} $A$ is the vertical restriction 
	$A_X Y = \nabla^v _X Y$
of the Levi-Civita connection $\nabla$, where $X$, $Y$ are horizontal vector fields on $\bb R^{n+k}$.
Moreover, a field is called \textit{basic} if it is both horizontal and projectable.

The aim of \cite{gromoll2001} is to prove Proposition \ref{prop:AXYparallel} below, thus concluding Theorem \ref{thm:GW} directly from Theorem 2.6 in \cite{gromoll1997metric} (see also the paragraph preceding Theorem 2.6 in \cite{gromoll1997metric}). 
After presenting two gaps in \cite{gromoll2001}, our goal is  to establish: 

\begin{prop}
\label{prop:AXYparallel}
For every basic  fields $X,Y$, $A_XY$ is parallel along $F$
\end{prop}

In the last paragraph, \cite{gromoll2001} show that $A_XY$ is parallel along $F$ for all basic  $X$, $Y$  if $A_xy$ is parallel along $F$ for all parallel horizontal $x$, $y$. 
The overall argument in \cite{gromoll2001} is then  to prove that $A_xy$ is parallel.

\begin{rem}
We remark that an argument similar to the one presented in section \ref{sec:3}  indeed shows that $A_xy$ is parallel, achieving Gromoll--Walschap's aim with a different approach. 
\end{rem}


Let $x,y$ be parallel horizontal fields along $F$.
In \cite[section 3]{gromoll2001}, a very interesting argument using the fiber volume form shows that  $\nabla_v(A_xy)=0$ for all $v\in im(A_x)+im(A_y)$. It follows that $im(A_x)$  defines an integrable distribution with totally geodesic leaves on $F$ (at least in the open and dense subset  where the rank of $im(A_x)$ is constant). 
The remainder of the proof deals with $\nabla_u(A_xy)$ for $u\in (im(A_x)+im(A_y))^\perp$ and can be divided in three steps:

\textbf{Step 1:} $im(A_x)$ defines a foliation  by affine subspaces on $F$
	
\textbf{Step 2:} $\pi$ is the composition of a linear projection $pr:\bb R^{n+k}\to \bb R^{n+k-l}$ followed by a Riemannian submersion $\pi:\bb R^{n+k-l}\to \bb R^n$, such that $\pi'$ is `fully twisted'. Specifically,  $TF'=pr(\bigoplus_{x\in \cal H} im(A_x))$, for $F'=pr(F)$, where $\cal{H}$ denotes the horizontal distribution along $F$

\textbf{Step 3:} The integrability tensor of $\pi'$ is parallel along $F'$

\vspace{0.2cm}

The gaps appear in Steps 1 and 3. In Step 1, a gap appears in arguing that $im(A_x)$ defines a Riemannian foliation on $F$. In Step 3, it seems to be implicitly assumed that $\bigoplus_{z\in \cal H} im(A_z)=im(A_x)+im(A_y)$ for a dense subset  $(x,y)\in \cal H\times\cal H$ along $F$. Although this statement is true for the homogeneous submersions in Theorem \ref{thm:GW}, one may believe that it generically  does not hold if $\dim( TF)>2\dim(\cal H)$.



\subsection{First gap}

For $p \in F$ let $T_p\bb R^{n+k} = \cal{H}_p + T_pF$ denote the orthogonal decomposition into the horizontal and the vertical space at $p$. For $x \in \cal{H}_p$, denote the adjoint of $A_x:\cal H_p\to T_pF$ by $A^*_x : T_pF \to \cal{H}_p$,  noting that $im(A_x)^\bot=\ker (A^*_x)$.

The next step in \cite{gromoll2001} was to prove that $im(A_x)$ defines a foliation by parallel affine subspaces.
This could be achieved by proving that, if $\gamma$ is a geodesic on $F$ satisfying $\gamma'(0)\in (im~A_x)^\perp=\ker (A^*_x)$, then $\gamma'(t)\in \ker (A^*_{x})$ for all $t$.
The first gap lies in the following claim (see \cite{gromoll2001}, section $3$). 

\begin{claim}\label{errorclaim}
For $a \in F$, let $x \in \mathcal{H}_{a}$ and  $u  \in ker (A^*_{ x})$. 
Then $A^*_{ x} \dot \gamma_{u} (t)=0$ for all $t$, where $\gamma_{u} (t) :=  a + tu$ is a line in $F$. 
\end{claim}

\begin{proof}[Discussion of the proof of Claim \ref{errorclaim}]
For $u \in ker (A^*_{x})$ we consider the variation $V$ on $[0,1] \times (-1,1)$, $V(t,s) := exp _{\gamma_u(s)}  (t  x) $ by horizontal geodesics which projects to the variation $W = \pi \circ V$ by geodesics on $M$.  
Likewise, since $V$ is by horizontal geodesics, its variational field $V_* D_s(t,0)$ is a Jacobi field that projects to the Jacobi field 
$Y(t) := (\pi \circ V) D_s(t,0) =  W_{\star} D_{s}(t,0)$  on $M$ induced by $W$. 
$Y$ satisfies $Y(0) = 0$ and 
\beam 
\label{Ynull}
\begin{array}{lcl}
\nonumber 
	Y'(0) &=& \nabla_{D_{t}(0,0)} ((\pi \circ V)_* D_{s}) 
	=  \pi_{\star} \nabla_{D_{t}(0,0)} (V_{\star} D_{s} )^{h} \\
	&=& - \pi_{\star} \nabla^{h}_{D_{t}(0,0)}(V_{\star}D_{s})^{v} 
	= \pi_{\star} A^{\star}_{x}u = 0,
\end{array}
\eeam
since $V_* D_s(0,0)=u$. 
The second equality follows since the fields $V_* D_t(t,0)$ and $(V_* D_s(t,0))^h$ are horizontal fields along $t \mapsto exp_{a}(tx)$.

The third equality is due to the identity 
\begin{multline*}
	\pi_{\star} \nabla^{h}_{D_{t}(0,0)} (V_{\star}D_{s})^{h} + \pi_{\star} \nabla^{h}_{D_{t}(0,0)} (V_{\star}D_{s})^{v}  
	= \pi_{\star} \nabla_{D_{t}(0,0)} (V_{\star}D_{s})\\ 
	= \pi_{\star} \nabla_{D_{s}(0,0)}( V_{\star} D_{t} ) 
	=  \pi_{\star} \nabla_{u} x = 0.
\end{multline*}
We follow $Y\equiv 0$ along $t \mapsto W(t,0)$.

At this point, $ x$ is stated to be  basic along $\gamma_u$.  
However, even though $W$ is a variation by geodesics emanating from a single point and the variational field $Y$ is trivial along the geodesic $t \mapsto W(t,0)$, 
this is not sufficient to imply that $ x$ is a basic field along $\gamma_u$. 
Indeed, one needs to show that $W_* D_s (t,s)=0$ for all $(t,s) \in [0,1] \times (-1,1)$. 
Then, $x$ is mapped to a single vector in $M$ since $t \mapsto W(t,s)$ corresponds to the geodesic $t\mapsto W(t,0)$ for all $s$. 
Hence, further arguments are required. The underlying issues can be seen in:
\begin{example}
Let $M = \bb R^2$ and $e_1$ and $e_2$ be the standard basis.  
Consider the variation $W : [0,1] \times (-1,1) \to \bb R^2$  given by $W(t,s) = t e_1 + s^2 e_2$ of the geodesic $t \mapsto t e_1$. 
Then its variational field $Y$ is trivial. But for $s\neq0$, $t \mapsto W(t,s)$ does not coincide with the geodesic $t \mapsto t e_1$.
\end{example}

\noindent On the other hand, if we assume that the field $x$ is  indeed basic along $\gamma_u$, then
\begin{equation*}
	A^{\star}_{x} \dot \gamma_{u} 
	= - \nabla^h _{\dot \gamma_u} x 
	= - \nabla ^{h} _{ u} ( x  \circ \gamma_{u} ) = 0,
\end{equation*}
 and Claim \ref{errorclaim} follows. 
\end{proof}

\subsection{Second gap}
Define the sets
\begin{equation*}
 	\cal A_p=span\{A_xy~|~x,y\in\cal H_p\}, \quad im(A) := \bigcup _{p \in F}  \mathcal{A}_{p}. 
\end{equation*}

According to Claim \ref{errorclaim}, the distributions $im{(A)}$ and~ $im({A})^{\perp}$ define an isometric splitting  $F\cong\bb R^l\times \bb R^{k-l}$, which extends to the whole ambient space, $ \bb R^{n+k} $, via holonomy transport. 
These properties result into a factorization of the projection $\pi$:

\begin{prop} \label{Aparallel} 
Assume that Claim \ref{errorclaim} is true. 
Then $\pi$ factors as an orthogonal projection $\bb R^{n+k-l} \times \bb R^{l} \rightarrow \bb R^{n+k-l} \times \{ 0\}$ followed by a Riemannian submersion $\pi' : \bb R^{n+k-l} \rightarrow M$. 
In particular, the fiber $F' := pr(F)$ is an affine subspace satisfying $TF' = im(A)$. 
\end{prop}

\begin{displaymath}
\xymatrix{  	F \ar[d]^{pr}\ar[r] 	&  \bb R^{n+k } \ar[d]_{pr} \ar[rd]^{\pi} 	& \\
	F' = im (A)\times \{0\}	\ar[r]&  \bb R^{n+k-l} \times \{0\}  \ar[r]^{\ \ \ \ \pi'} 	&M }
\end{displaymath}

\noindent What we therefore obtain is a Riemannian submersion $\pi'$ which only contains the 'twisting part' of the former submersion $\pi$.  
Although $F'$ is spanned by integrability fields, the induced metric foliation $\mathcal{F}'$ of $\pi'$ is not necessarily \textit{substantial} along $F'$. 
That is, one can not guarantee that there is a single horizontal $x\in (T_pF')$ such that $im(A_x)=T_pF'$. 

This observation is relevant since the concluding argument in \cite{gromoll2001}, in the proof that $A_xy$ is parallel, seems to be based on the substantiality of $\cal F'$: recall that $\nabla_v(A_xy)=0$ for all $v\in im(A_x)+im(A_y)$. If $\cal F'$ is substantial, there is an open and dense set of horizontal vectors $z\in\cal H_p$, $p\in F'$, such that $im(A_z)=T_pF'$. In particular,  it would follow that $\nabla_v(A_xy)=0$ for all $~x,y\in\cal H$ and $v\in TF'=im(A_x)=im(A)$. Otherwise, $A_xy$ could be only parallel on $im(A_x)+im(A_y)$, but not on the whole $im(A)$.
More specifically, the following statement in \cite{gromoll2001} lacks a proof:


\begin{claim}\label{missarg} 
An argument similar to the one that led to \cite[Lemma 2.4]{gromoll2001} implies that each  field $A_{x} y$ is parallel along the fiber $F'$ of $\pi'$ for parallel horizontal fields $x$ and $y$  along $F'$.
\end{claim}

\cite[Lemma 2.4]{gromoll2001} only proves that $\nabla_v(A_xy)=0$ for  $v\in im(A_x)+im(A_y)$ (it is restated in the next section). 

\section{$A_XY$ is parallel}\label{sec:3}

From now on, we fix basic  fields  $X,Y$ along $F$. We directly prove that $A_XY$ is parallel along $F$, avoiding Claim \ref{errorclaim} and Proposition \ref{Aparallel}.
We start just before the first gap by recalling from \cite{gromoll2001} that:

\begin{lemma}[\cite{gromoll2001}, Lemma 2.4]\label{lem:GW}
Let $p\in F$. Then, $(\nabla_vA)_XY=0$ for all $X,Y\in \cal H_p$ and $v\in {im}(A_X)+{im}(A_Y)$.
\end{lemma}

In order to prove that $A_XY$ is parallel, we follow the proof of Theorem 2.6 in \cite{gromoll1997metric} and show that $A_X Y$ is the gradient of a function $f:F\to \bb R$. As in \cite{gromoll1997metric}, we recall that constant length gradients in affine spaces are parallel and that $\|A_XY\|$ is constant (as it follows from O'Neill's equation $3\|A_XY\|^2=R_M(\pi_ *X,\pi_*Y,\pi_*Y,\pi_*X)$).

Denote $\cal I={im}(A_X)+{im}(A_Y)$, thus $\cal I^\bot=\ker{(A_X^*)}\cap \ker (A_Y^*)$.

\begin{lemma}\label{lem1} 
	For every $u,u'\in \cal I^\bot$ and $v\in\cal I$:
	\begin{enumerate}[$(i)$]
		\item $\nabla_v(A_XY) = (\nabla^v_YS)_Xv-(\nabla_X^vS)_Yv$
		\item $\lr{\nabla_u(A_XY),v} = 0$
		\item $\lr{\nabla_v(A_XY),u} = 0$
		\item $\lr{\nabla_u(A_XY),u'} = \lr{-(\nabla^v_{X}S)_Yu,u'}$
	\end{enumerate}
\end{lemma}

\begin{proof}
	O'Neill's equation (see, e.g., \cite[page 43]{gw}) gives
	\begin{equation}\label{eq:ONeill0}
		(\nabla^v_wA)_XY = -A_YA_X^*w-(\nabla_X^vS)_Yw,
	\end{equation}
	for all $w\in TF$. 
	Therefore, Lemma \ref{lem:GW} gives 
	\begin{equation}\label{eq:ONeill2}
	A_YA_X^*v=-(\nabla_X^vS)_Yv
	\end{equation} 
	for all $v\in\cal I$. 
	Item $(i)$ now follows from a straightforward computation: 
	\begin{align}
\nonumber
		 \nabla_v(A_XY) &= (\nabla_v^vA)_XY+A_YA^*_Xv-A_XA_Y^*v \\
\label{eq:ONeill1} &= -(\nabla_X^vS)_Yv-A_XA_Y^*v\\
	\nonumber  \label{eq:ONeill} 	&= -(\nabla_X^vS)_Yv+(\nabla_Y^vS)_Xv,	
	\end{align}
	\noindent where equation \eqref{eq:ONeill1} is valid for all $w\in TF$ and the last equality follows from \eqref{eq:ONeill2}.

	For item $(ii)$, we get
	\beam 
	\begin{array}{lcl}
	\nonumber 
		\lr{\nabla_u(A_XY),v} &=& -\lr{A_XA_Y^*u+(\nabla^v_XS)_Yu,v} \\
		&=& -\lr{(\nabla^v_XS)_Yu,v} 
		= -\lr{u,(\nabla^v_XS)_Yv} \\
		&=& \lr{u,A_YA_X^*v}=0.	
	\end{array}
	\eeam
	The first and fourth equalities follow from equations \eqref{eq:ONeill1} and \eqref{eq:ONeill2}, respectively, and the last since $u\perp {im}(A_X)$. 
Item $(iii)$ follows from item $(i)$ and equation \eqref{eq:ONeill2}, since $u\in \cal I^\bot$.
	Item $(iv)$ follows from equation \eqref{eq:ONeill1} and since $u \in \cal I^\bot$.
\end{proof}

\begin{prop}
There is a function $f:F\to \bb R$ whose gradient is $A_XY$.
\end{prop}

\begin{proof}
	Consider the $1$-form $\alpha: TF\to \bb R$, $\alpha(u)=\lr{A_XY,u}$. 
	Then $\alpha=df$ for some $f$ if and only if  
	\[d\alpha(u,v)=\lr{\nabla_u(A_XY),v}-\lr{\nabla_v(A_XY),u}=0 \]
	for all $u,v\in TF$. 
	But the latter holds by a straightforward computation by distinction of cases for $u, v \in TF=\cal I+\cal I^\bot$, using Lemma \ref{lem1} and by observing that $(\nabla_X^vS)_Y:TF\to TF$ is a symmetric operator since $S_Y$ is symmetric. 
\end{proof}

\bibliographystyle{alpha}
\bibliography{main}

\end{document}